\documentclass[a4paper, 10pt, twoside, notitlepage]{amsart}

\usepackage{amsmath,amscd}
\usepackage{amssymb}
\usepackage{amsthm}
\usepackage{comment}
\usepackage{graphicx, xcolor}

\usepackage{mathrsfs}
\usepackage[ocgcolorlinks, linkcolor=blue]{hyperref}

\usepackage{bm}
\usepackage{bbm}
\usepackage{url}

\usepackage[utf8]{inputenc}
\usepackage{mathtools,amssymb}
\usepackage{esint}
\usepackage{tikz}
\usepackage{dsfont}
\usepackage{relsize}
\usepackage{url}
\usepackage{xcolor}
\usepackage{graphicx}
\usepackage{mathrsfs}
\usepackage[shortlabels]{enumitem}
\usepackage{lineno}
\usepackage{amsmath}
\usepackage{enumitem}
\usepackage{amsthm} 
\usepackage{verbatim}
\usepackage{dsfont}
\numberwithin{equation}{section}

\allowdisplaybreaks

\mathtoolsset{showonlyrefs}

\graphicspath{{images/}}

\newtheorem{theorem}{Theorem}[section]
\newtheorem{lemma}[theorem]{Lemma}

\newtheorem{corollary}[theorem]{Corollary}

\newtheorem{proposition}[theorem]{Proposition}

\newtheorem{assumption}{Assumption}
\newtheorem{remark}[theorem]{Remark}

\title[Nonlocal elliptic equations]{Determining both leading coefficient and source in a nonlocal elliptic equation}

\author[Y.-H. Lin]{Yi-Hsuan Lin}
\address{Department of Applied Mathematics, National Yang Ming Chiao Tung University, Hsinchu, Taiwan}
\email{yihsuanlin3@gmail.com}

\newcommand{\R}{{\mathbb R}}

\newcommand{\N}{{\mathbb N}}

\newcommand {\p} {\partial}

\newcommand{\LC}{\left(}
\newcommand{\RC}{\right)}
\newcommand{\wt}{\widetilde}

\newcommand{\norm}[1]{\lVert #1 \rVert}
\newcommand{\abs}[1]{\left\lvert #1 \right\rvert}
\DeclareMathOperator{\supp}{supp} 

\begin{document}

	\maketitle
	\begin{abstract}
		
	 In this short note, we investigate an inverse source problem associated with a nonlocal elliptic equation 
	$\LC -\nabla \cdot \sigma \nabla \RC^s u =F$ that is given in a bounded open set $\Omega\subset \R^n$, for $n\geq 3$ and $0<s<1$. We demonstrate both $\sigma$ and $F$ can be determined uniquely by using the exterior Dirichlet-to-Neumann (DN) map in $\Omega_e:=\R^n\setminus \overline{\Omega}$. The result is intriguing in that analogous theory cannot be true for the local case generally, that is, $s=1$. The key ingredients to prove the uniqueness is based on the unique continuation principle for nonlocal elliptic operators and the reduction from the nonlocal to the local via the Stinga-Torrea extension problem.
		
		\medskip
		
		\noindent{\bf Keywords.} Nonlocal elliptic operators, the Calder\'on problem, inverse source problem, simultaneous determination, Caffarelli-Silvestre extension, Stinga-Torrea extension.
		
		\noindent{\bf Mathematics Subject Classification (2020)}: 35R30; 26A33; 35J70

	\end{abstract}

	\tableofcontents

	\section{Introduction}\label{sec: introduction}

	Inverse problems for fractional partial differential equations (PDEs) have received plentiful studies in recent years. The research of fractional inverse problems was first initiated in \cite{GSU20}, where the authors recover the potential by using the exterior \emph{Dirichlet-to-Neumann} (DN) map. Moreover, in the work \cite{GSU20}, the authors constructed useful tools: \emph{Unique continuation property} (UCP) and \emph{Runge approximation property} for the fractional Laplacian $(-\Delta)^s$, for $s\in (0,1)$. Based on these two remarkable results, one may solve several challenging problems in the nonlocal setting, even they remain open for the local case $s=1$. For instance, one can determine both drift and potential terms for a fractional Schr\"odinger equation by the exterior DN map (see \cite{cekic2020calderon}). A simultaneous determination for both unknown obstacle and surrounding medium was investigated in \cite{CLL2017simultaneously}. Variable coefficients fractional Calder\'on problem was considered in \cite{GLX}.  In addition, if-and-only-if monotonicity relations are proved by \cite{harrach2017nonlocal-monotonicity,harrach2020monotonicity}. These problems are either not true or open for the local case $s=1$. So far, most of literature of nonlocal inverse problems are to recover coefficients "outside" a given nonlocal operator, such as \cite{bhattacharyya2021inverse,CMR20,CMRU20,GLX,CLL2017simultaneously,GRSU18,lin2020monotonicity,LL2020inverse,LL2022inverse,LLR2019calder,KLW2021calder,RS17}.

     Recently, the recovery of leading coefficients in nonlocal operators has been studied. In the works \cite{feizmohammadi2021fractional,feizmohammadi2021fractional_closed,ruland2023revisiting}, the authors use the local source-to-solution map to determine the metric (up to diffeomorphism) on closed Riemannian manifolds for the anisotropic fractional Calder\'on problem under slightly different regularity assumptions. More concretely, the local source-to-solution map for this type problem provides a connection between the associated Dirichlet Poisson kernel, heat kernel and the wave kernel. On the other hand, one can also consider the exterior DN map as a measurement. In the works \cite{CGRU2023reduction,LLU2023calder}, the authors introduces novel reduction formula for both fractional elliptic and parabolic equations via the \emph{Caffarelli-Silvestre} type extension. More concretely, the reduction will deduce that the nonlocal DN maps determines the corresponding local DN map. In this work, we apply this idea so that we are able to determine both leading coefficient and source by using the exterior DN map, which cannot be true for the case $s=1$ in general. Similar results have been also investigated in \cite{GU2021calder,LLU2022para} by utilizing different methods. Last but not least, we also point out that there exist several works related fractional inverse problems but in different formulation, such as \cite{RZ2022unboundedFracCald,RZ2022FracCondCounter,CRZ2022global,RZ-low-2022,CRTZ-2022,zimmermann2023inverse,LZ2023unique}. In these works, the leading nonlocal operators have been considered as the combination of fractional divergence and fractional gradient.
	
	\bigskip 
	
   \noindent $\bullet$ \textbf{Mathematical formulation.} Let $\Omega \subset \R^n$ be a bounded open set with Lipschitz boundary $\p \Omega$ for $n\geq 3$, and $s\in (0,1)$.
    Let $\sigma\in C^\infty(\R^n)$ be a scalar function satisfying 
    \begin{equation}\label{ellipticity}
    	\lambda \leq \sigma(x)\leq \lambda^{-1}, \text{ for }x\in \R^n,
    \end{equation}
    for some constant $\lambda\in (0,1)$.
     Given $F\in L^2(\Omega)$, consider the nonlocal elliptic equation 
    \begin{equation}\label{equ: main}
    	\begin{cases}
    		 (-\nabla \cdot \sigma \nabla )^s  u =F &\text{ in }\Omega, \\
    		u=f &\text{ in }\Omega_e,
    	\end{cases}
    \end{equation}
	where $\Omega_e := \R^n\setminus \overline{\Omega}$. The exterior value problem \eqref{equ: main} is well-posed by using a classical argument (via the Lax-Milgram) demonstrated in \cite[Section 3]{GLX}. More precisely, given $f\in \wt H^s(\Omega_e)$, there exists a unique solution $u\in H^s(\R^n)$ to \eqref{equ: main}. We will introduce fractional Sobolev spaces in Section \ref{sec: pre}.
    Since the equation \eqref{equ: main} is well-posed, let $W\subset \Omega_e$ be an open set with Lipschitz boundary, then we can define the (partial) DN map $\Lambda^s_{\sigma,F}$ formally by 
    \begin{equation}\label{DN map}
    	\Lambda^s_{\sigma,F}:\wt H^s(W) \to H^{-s}(W) \quad  f\mapsto \left.\LC -\nabla \cdot \sigma \nabla \RC^s u_f \right|_{W},
    \end{equation}
where $u_f$ is the unique solution to \eqref{equ: main}.

\begin{enumerate}[\textbf{(IP)}]
	\item \label{IP} \textbf{Inverse Problem.} Can we determine both $\sigma$ and $F$ in $\Omega$ by using the DN map $\Lambda_{\sigma,F}^s$?
\end{enumerate}

Before answer the above question, let us first revisit the case $s=1$, i.e., the local case. Let us consider the second order elliptic equation 
\begin{equation}\label{equ local}
 \begin{cases}
 	-\nabla \cdot \LC \sigma \nabla v \RC =F &\text{ in }\Omega, \\
 	v =g&\text{ on }\p \Omega.
 \end{cases}
\end{equation}
Let $\Lambda_{\sigma,F}$ be the (local) DN map of \eqref{equ local}, which is given by
\begin{equation}\label{DN map local}
	\Lambda_{\sigma,F}: H^{1/2}(\p \Omega)\to H^{-1/2}(\p \Omega), \quad g\mapsto \left. \sigma \nabla v_g \cdot \nu \right|_{\p \Omega},
\end{equation}
where $v_g\in H^1(\Omega)$ is the solution to \eqref{equ local}, and $\nu$ denotes the unit outer normal on $\p \Omega$.
We are also interested that whether one can determine $\sigma$ and $F$ in $\Omega$ simultaneously. Unfortunately, even we assume $\sigma$ is known a priori, we cannot determine both $\sigma$ and $F$ by using the local DN map \eqref{DN map local}. More concretely, let us consider an arbitrary function $\varphi \in C^2_c(\Omega)$, so that 
$$
\varphi|_{\p \Omega} = \left. \sigma \p_\nu \varphi\right|_{\p\Omega}=0.
$$ 
Then the function $\wt v:=v+\varphi$ satisfies 
\begin{equation}\label{equ local 1}
	\begin{cases}
	-\nabla \cdot \LC \sigma \nabla \wt v\RC = F -\nabla \cdot \LC \sigma\nabla \varphi \RC & \text{ in }\Omega,\\
\wt v=v+\varphi=g  &\text{ on }\p \Omega.
	\end{cases}
\end{equation}
By the condition $\varphi \in C^2_c(\Omega)$, one can see that the DN map of \eqref{equ local} and \eqref{equ local 1} are the same, but the source term $\nabla \cdot\LC \sigma \nabla \varphi\RC$ appears in the right hand side of \eqref{equ local 1} may not be zero in $\Omega$, which implies the source $F$ and $F-\nabla \cdot\LC \sigma \nabla \varphi\RC$ might not  equal generally. Therefore, this means the source function $F$ cannot be determined due to this natural gauge invariance even if the leading coefficient $\sigma$ is given a priori. Moreover, in recent papers, the authors and I considered how to break the gauge for both (local) \emph{semilinear} elliptic and parabolic equations, and we refer readers to \cite{LL2022uniqueness,KLL2023determining} for more details.

Hence, it is natural to ask similar problem \ref{IP} in the nonlocal setting. To solve our problem, let us pose an additional assumption:
\begin{assumption}\label{Hypothesis} 
	We assume $\sigma \in C^\infty(\R^n)$ satisfying the ellipticity \eqref{ellipticity} and $\sigma=1$ in $\Omega_e$.
\end{assumption}	

Surprisingly, unlike the local case (i.e., $s=1$), we can determine both $\sigma$ and $F$ uniquely by using the DN map $\Lambda^s_{\sigma,F}$, and our main result is stated as follows.
	
	\begin{theorem}\label{thm: main}
	 For $n\in \N$, $n\geq 3$, let $\Omega , W \subset \R^n$ be bounded open sets with Lipschitz boundaries such that $\overline{\Omega}\cap \overline{W}=\emptyset$, and $s\in (0,1)$. Let $\sigma_j \in C^\infty(\R^n)$ satisfy Assumption \ref{Hypothesis}, $F_j\in L^2(\Omega)$, and $\Lambda^s_{\sigma_j,F_j}$ be the DN map of 
	\begin{equation}\label{equ: main j=12}
		\begin{cases}
		 (-\nabla \cdot \sigma_j \nabla )^s u_j=F_j &\text{ in }\Omega, \\
			u_j=f &\text{ in }\Omega_e,
		\end{cases}
	\end{equation}
	for $j=1,2$. Suppose 
	\begin{equation}\label{same DN map}
		\Lambda^s_{\sigma_1,F_1}(f)=\Lambda^s_{\sigma_2,F_2}(f) \text{ in }W, \text{ for any }f\in C^\infty_c(W),
	\end{equation}
	then we have $\sigma_1=\sigma_2$ and $F_1=F_2$ in $\Omega$.
	\end{theorem}
	
\begin{remark}
	Let us emphasize that Assumption \ref{Hypothesis} is only used in the proof of Theorem \ref{thm: main}.
\end{remark}

	The paper is organized as follows. In Section \ref{sec: pre}, we review fractional Sobolev space and some properties of nonlocal elliptic operators. In Section \ref{sec: CS extension}, we recall the Cafffarelli-Silvestre/Stinga-Torrea extension problem and associated results, which will be used in our proof. We also offer a comprehensive discussion for the corresponding local inverse source problem. Finally, we prove our main theorem in Section \ref{sec: proof}.

	\section{Preliminaries}\label{sec: pre}

	In this section, we review several useful materials for our study.
	
	\subsection{Fractional Sobolev spaces}
	
	Fractional Sobolev spaces are widely studied in many literature, and we refer readers to \cite{DNPV12} for more details. In what follows, we simply recall fundamental properties for these spaces.
	For $s\in (0,1)$, let $H^s(\R^n)=W^{s,2}(\R^{n})$ be the $L^2$-based fractional Sobolev space under the norm 
	\begin{equation*}
		\|u\|_{H^{s}(\mathbb{R}^{n})}:=\|u\|_{L^{2}(\R^{n})}+[u]_{H^{s}(\R^{n})}\label{eq:NormHs}
	\end{equation*}
	where 
	\[
	[u]_{H^{s}(A)}:=\LC\int_{A\times A}\frac{\left|u(x)-u(y)\right|^{2}}{|x-y|^{n+2s}}\, dxdy\RC^{1/2},
	\]
	denotes the seminorm, for any open set $A\subset \R^{n}$.
	
	As in \cite{GSU20}, let $A$ be a nonempty open set of $\mathbb{R}^{n}$, the space $C_{c}^{\infty}(A)$ consists all $C^{\infty}(\mathbb{R}^{n})$-smooth functions supported in $A$. For $a\in \R$, we use the following notations 
	\begin{align*}
		H^{a}(A) & :=\left\{u|_{A}: \, u\in H^{a}(\R^{n})\right\},\\
		\wt H^{a}(A) & :=\text{closure of \ensuremath{C_{c}^{\infty}(A)} in \ensuremath{H^{a}(\R^{n})}},\\
		H_{0}^{a}(A) & :=\text{closure of \ensuremath{C_{c}^{\infty}(A)} in \ensuremath{H^{a}(A)}},
	\end{align*}
	to denote various fractional Sobolev spaces.
	Meanwhile, the space $H^{a}(A)$ is complete in the sense
	\[
	\|u\|_{H^{a}(A)}:=\inf\left\{ \|v\|_{H^{a}(\mathbb{R}^{n})}: \, v\in H^{a}(\mathbb{R}^{n})\mbox{ and }v|_{A}=u\right\} .
	\]
	In particular, as $a=s\in (0,1)$, let us denote $H^{-s}(A)$ to be the dual space of $\wt H^s(A)$, so that $H^{-s}(A)$ can be characterized by  
	\[
	H^{-s}(A) = \left\{ u|_{A} : u \in H^{-s}(\mathbb{R}^{n}) \right\} \quad \text{with} \quad \inf_{v\in H^{s}(\mathbb{R}^{n}), \ v|_{A}=u} \| v \|_{H^{s}(\mathbb{R}^{n})},
	\]
	
	\subsection{Nonlocal elliptic operators}
	
	Let $\mathcal{L}:=-\nabla \cdot \LC \sigma \nabla \RC$ be a second order uniformly elliptic operator, then  the nonlocal elliptic operator $\mathcal{L}^s= (-\nabla \cdot \sigma \nabla)^s$ can be defined via the forthcoming procedures.
	Let us first recall the representation formula 
	$$
	\lambda^{s}=\dfrac{1}{\Gamma(-s)}\int_{0}^{\infty}\LC e^{-t\lambda}-1\RC \frac{dt}{t^{1+s}},
	$$
	for $s\in(0,1)$, where $\Gamma(-s):=-\frac{\Gamma(1-s)}{s}$. Here $\Gamma(\cdot )$
	denotes the Gamma function. Hence, for $s\in(0,1)$, the nonlocal operator $\mathcal{L}^s$ can be defined by 
	\begin{equation}
		\mathcal{L}^{s}:=\int_{0}^{\infty}\lambda^{s}\,dE_{\lambda}=\frac{1}{\Gamma(-s)}\int_{0}^{\infty}\left(e^{-t\mathcal{L}}-\mbox{Id}\right)\,\frac{dt}{\tau^{1+s}},\label{eq:1111}
	\end{equation}
	where $\mathrm{Id}$ stands for the identity map.
	Here 
	\begin{equation}
		e^{-t\mathcal{L}}:=\int_{0}^{\infty}e^{-t \lambda}\,dE_{\lambda}\label{eq:heat-semigroup}
	\end{equation}
	is a bounded self-adjoint operator in $L^{2}(\mathbb{R}^{n})$ for
	each $t \ge0$, where $\left\{ E_\lambda \right\}$ denotes the spectral resolution of $\mathcal{L}$, and each $E_\lambda$ is a projection in $L^2(\R^n)$. The family $\left\{e^{-t\mathcal{L}}\right\}_{t \ge0}$
	is the heat semigroup corresponding to $\mathcal{L}=-\nabla \cdot \LC \sigma \nabla \RC$.

	Moreover, let us recall the unique continuation principle (UCP in short) for $\mathcal{L}^s$, which was proved in \cite[Theorem 1.2]{GLX}.
	
	\begin{proposition}[Unique continuation principle]\label{Prop: UCP}
		For $n\in \N$, let $\sigma\in C^\infty(\R^n)$ satisfy \eqref{ellipticity}. Let $u\in H^s(\R^n)$ fulfill $u=\mathcal{L}^su=0$ in some open set $W\subset \R^n$. Then $u\equiv 0 $ in $\R^n$.
	\end{proposition}

	Note that the UCP holds for general smooth coefficients $\sigma$ defined in $\R^n$, and we do not need to use Assumption \ref{Hypothesis} in the proof of Proposition \ref{Prop: UCP}.

	\subsection{The forward problem and DN map}
	
	We want to study the well-posedness of \eqref{equ: main}. Applying the results in \cite[Section 3]{GLX}, it is known that the exterior problem \eqref{equ: main} is well-posed, and the DN map can be defined by the bilinear form 
	\begin{equation}\label{bilinear}
      \begin{split}
      	B(u, w)&:=\left\langle \mathcal{L}^s u, w \right\rangle  \\
      	&\ = \frac{1}{2}\int_{\R^n\times \R^n} \LC u(x)-u(z)\RC \LC w(x)-w(z) \RC K_s(x,z)\, dxdz,
      \end{split}
	\end{equation}
	where $K_s(x,z)$ is given by
	\begin{equation}
		K_s(x,z):=\frac{1}{\Gamma(-s)} \int_0^\infty p_t(x,z)\frac{dt}{t^{1+s}}.
 	\end{equation}
	Here $p_t(x,z)$ denotes the symmetric heat kernel for $\mathcal{L}$ satisfying 
	\begin{equation}
		\LC e^{-t\mathcal{L}} f\RC (x)=\int_{\R^n} p_t(x,z)f(z)\, dz, \text{ for }x\in \R^n, \ t>0,
	\end{equation}
	i.e., 
	\begin{equation}
	 \begin{cases}
	 	\LC \p _t+\mathcal{L}\RC  \LC e^{-t\mathcal{L}} f \RC   =0 &\text{ for }(x,t)\in \R^{n} \times (0,\infty), \\
	 	\LC e^{-t\mathcal{L}} f\RC (x,0)=f(x) &\text{ for }x\in \R^n.
	 \end{cases}
	\end{equation}
	Recall that the heat kernel $p_t(x,z)$ admits the pointwise estimate (for instance, see \cite{Davies90})
	\begin{equation}\label{heat kernel estimate}
		c_1 e^{-\alpha_1\frac{|x-z|^2}{t}}t^{-\frac{n}{2}}\leq p_t (x,z)\leq c_2 e^{-\alpha_2\frac{|x-z|^2}{t}}t^{-\frac{n}{2}}, \text{ for }x,z\in \R^n,
	\end{equation}
	for some positive constants $c_1,c_2,\alpha_1$ and $\alpha_2$. With the aid of pointwise estimate \eqref{heat kernel estimate} for heat kernel, we can obtain pointwise estimate for $K_s(x,z)$ such that 
	\begin{equation}
		\frac{C_1}{\abs{x-z}^{n+2s}}\leq K_s(x,z)\leq \frac{C_2}{\abs{x-z}^{n+2s}}, \text{ for }x,z\in \R^n,
	\end{equation}
	for some constants $C_1,C_2>0$.
	The above estimate will help us to show the well-posedness of \eqref{equ: main} (see \cite[Section 3]{GLX} for detailed arguments).

	Therefore, the exterior problem \eqref{equ: main} is well-posed, then the bilinear form \eqref{bilinear} can be used to defined the DN map
	\begin{equation}
		\left\langle \Lambda_{\sigma,F}^s f, g\right\rangle = B\LC u_f, w_g\RC,
	\end{equation}
	where $u_f\in H^s(\R^n)$ is the solution to \eqref{equ: main} and $w_g\in H^s(\R^n)$ can be any representative of $g$. Using a duality characterization, one can see that the DN map \eqref{DN map} is defined rigorously.

	\section{Local equations}\label{sec: CS extension}
	
	Let us consider the extension problem and related properties in this section.

	\subsection{The extension problem}
	
	By the pioneer work of \cite{CS07}, it is known that the fractional Laplacian can be characterized by a degenerate extension problem with $A_2$-Muckenhoupt weight. This is called the \emph{Caffarelli-Silvestre extension} in the literature. Moreover, the result was extended to general (variable) coefficients elliptic operator by Stinga-Torrea \cite{ST10}, so that the nonlocal operator $(-\nabla \cdot \sigma \nabla)^s $ can also be characterized by utilizing the extension problem 
	  \begin{equation}\label{extension problem}
	  	 \begin{cases}
	  	 	\nabla_{x,y}\cdot \LC y^{1-2s}\wt \sigma \nabla_{x,y}\wt u \RC =0 & \text{ in }\R^{n+1}_+, \\
	  	 	\wt u(x,0)=u(x) &\text{ in }\R^n,
	  	 \end{cases}
	  \end{equation}
	  where $\wt \sigma$ is written as  
	  	\begin{align}\label{tilde sigma(x)}
	  	\wt \sigma (x)=\left( \begin{matrix}
	  		\sigma(x)\mathrm{I}_n& 0\\
	  		0 & 1 \end{matrix} \right).
	  \end{align}
	 Here  $\mathrm{I}_n$ is  $n\times n$ identity matrix. Then the nonlocal elliptic operator $(-\nabla \cdot \sigma \nabla )^s$ can be written as  
	  \begin{equation}
	  	-\lim_{y\to 0}y^{1-2s}\p_y \wt u =d_s (-\nabla \cdot \sigma \nabla)^s u(x) \text{ in }\R^n,
 	  \end{equation}
	for some constant $d_s$ depending only on $s\in (0,1)$.
	
	By using \cite[Lemma 6.2]{CGRU2023reduction}, it is known that solutions $\wt u$ of the extension problem \eqref{extension problem} has nice decay properties. For the sake of completeness, let us recall the result as follows.

	\begin{lemma}\label{Lemma: decay estimate}
		Let $\Omega, W\subset \R^n$ be bounded open sets with Lipschitz boundaries, for $n\in \N$. Given $s\in (0,1)$, let $u\in H^s(\R^{n})$ with $\supp(u)\subset \overline{\Omega\cup W}$, and $\wt u$ be the corresponding solution of the extension problem \eqref{extension problem}. Then the function $\wt u(x,y)$ satisfies the following decay estimates:
		\begin{enumerate}[(a)]
			\item\label{item a decay x} \begin{equation}\label{decay estimate in x}
				\begin{split}
					\left| \wt u  (x,y)\right| \leq C y^{-n}\norm{u}_{L^1(\R^n)}, \quad \left| \nabla \wt u(x,y)\right| \leq C y^{-n-1}\norm{u}_{L^1(\R^n)}
				\end{split}
			\end{equation}
			for any $(x,y)\in \R^{n+1}_+:=\R^n \times (0,\infty)$, where $C>0$ is a constant independent of $u$ and $y>0$.

			\item\label{item c decay y} For $1\leq p,q,r \leq \infty$ with $1+\frac{1}{r}=\frac{1}{p}+\frac{1}{q}$, $\wt u$ satisfies 
			\begin{equation}\label{decay estimate in y}
				\begin{split}
					\left\|  \wt u(\cdot,y)\right\|_{L^r(\R^{n})} \leq C y^{\frac{n}{p}-n}\norm{u}_{L^q(\R^n)},
				\end{split}
			\end{equation}
			and 
			\begin{equation}\label{decay estimate deri}
				\left\|\nabla_{x,y}\wt u(\cdot , y)	\right\|_{L^r(\R^{n})}  \leq C y^{\frac{n}{p}-n-1}\norm{u}_{L^q(\R^n)},
			\end{equation}
			where $C>0$ is a constant independent of $u$ and $y>0$.
			
		\end{enumerate}
	\end{lemma}	
	
	Similar decay results have been investigated in the nonlocal parabolic case in our earlier work \cite{LLU2023calder}.

	\subsection{Relation between the nonlocal and local}\label{subsection:local-nonlocal}
	
	Motivated by \cite{CGRU2023reduction,LLU2023calder,ruland2023revisiting}, one can use the extension problem to find some equations to prove our result.
	Let us consider an auxiliary function 
	\begin{equation}
		v(x):=\int_0^\infty y^{1-2s}\wt u(x,y)\, dy,
	\end{equation}
	where $\wt u(x,y)$ solves \eqref{extension problem}.
	By using \eqref{extension problem}, we can obtain 
	\begin{equation}
		\begin{split}
			0&= \int_0^\infty  y^{1-2s}\nabla \cdot \LC \sigma \nabla \wt u \RC dy + \int_0^\infty \p _y \LC y^{1-2s}\p_y \wt u \RC dy  \\
			&= \nabla \cdot \LC \sigma \nabla v \RC +\lim_{y\to \infty}y^{1-2s}\p_y \wt u -\lim_{y\to 0}y^{1-2s}\p_y \wt u \\
			 &= \nabla \cdot \LC \sigma \nabla v \RC  -\lim_{y\to 0}y^{1-2s}\p_y \wt u,
		\end{split}
	\end{equation}
	where we used the decay estimate \eqref{decay estimate deri} as $r=\infty$ and $p=q=2$ such that $\lim_{y\to \infty}y^{1-2s}\p_y \wt u=0$.
	This implies
	\begin{equation}
		\begin{split}
			-\nabla \cdot \LC \sigma \nabla v \RC  = -\lim_{y\to 0}y^{1-2s}\p_y \wt u =d_s\LC -\nabla \cdot \sigma \nabla \RC^s u \text{ in }\R^n,
		\end{split}
	\end{equation}
	Particularly, combining \eqref{equ: main} in the region $\Omega$, we obtain 
     \begin{equation}\label{equ: nonlocal to local}
        \begin{cases}
        		-\nabla \cdot \LC \sigma \nabla v \RC  =d_s F & \text{ in }\Omega, \\
        		v=\left. \int_0^\infty y^{1-2s}\wt u_f(x,y) \, dy\right|_{\p \Omega} &\text{ on }\p\Omega.
        \end{cases}
     \end{equation}
     As we mentioned before, we want to determine $\sigma$ and $F$ simultaneously.

	In particular, as $F\equiv 0$ in $\Omega$, it is known that the nonlocal DN map determines the local DN map. More concretely, by \cite[Theorem 1]{CGRU2023reduction} shows that the nonlocal DN map $\Lambda_{\sigma,0}^s$ determines the local DN map $\Lambda_{\sigma,0}$. Hence, if we consider the zero exterior data $f=0$ into \eqref{equ: main}, and we denote the solution as $u^{(0)}\in H^s(\R^n)$, then we have 
	\begin{equation}\label{equ: main f=0}
		\begin{cases}
			(-\nabla \cdot \sigma \nabla )^s  u^{(0)} =F &\text{ in }\Omega, \\
			u^{(0)}=0 &\text{ in }\Omega_e.
		\end{cases}
	\end{equation}
	Next, by taking the function $\mathbf{u}^{(f)} :=u^{(f)}-u^{(0)}$, where $u^{(f)} , u^{(0)}\in H^s(\R^n)$ are the solutions to \eqref{equ: main} and \eqref{equ: main f=0}, respectively, then $\mathbf{u}^{(f)} \in H^s(\R^n)$ is the solution to
	\begin{equation}\label{equ: main F=0}
		\begin{cases}
			(-\nabla \cdot \sigma \nabla )^s  \mathbf{u}^{(f)} =0 &\text{ in }\Omega, \\
			\mathbf{u}^{(f)}=f &\text{ in }\Omega_e.
		\end{cases}
	\end{equation}
	Let $\mathbf \Lambda^s_{\sigma}$ nonlocal DN map of the above equation, and one can see that 
	\begin{equation}
	\mathbf \Lambda^s_{\sigma}(f)\equiv \Lambda^s_{\sigma,0}(f) =	\Lambda^s_{\sigma,F}(f) -	\Lambda^s_{\sigma,F}(0), 
	\end{equation}
	where the right hand side of the above identity is given a priori since we know the information of nonlocal DN map $\Lambda^s_{\sigma,F}$. Applying \cite[Theorem 1]{CGRU2023reduction}, one can conclude that $\mathbf \Lambda_{\sigma}^s$ determines the local DN map $\mathbf \Lambda_\sigma$, where $\mathbf \Lambda_\sigma$ is the DN map of 
	\begin{equation}\label{equ: local F=0}
		\begin{cases}
			-\nabla \cdot \LC  \sigma \nabla   \wt v\RC  =0 &\text{ in }\Omega, \\
			\wt v=g &\text{ on }\p \Omega.
		\end{cases}
	\end{equation}
	Since $\sigma$ is a scalar function satisfies Assumption \ref{Hypothesis} so that $\sigma_1=\sigma_2=1$ on $\p \Omega$, by using the seminal work of the global uniqueness of \eqref{equ: local F=0} (for example, see \cite{SU87}), then the conductivity $\sigma$ in the equation \eqref{equ: local F=0} can be determined uniquely by its DN map. Combining with previous discussions and explanations in Section \ref{sec: introduction}, we can immediately conclude the next result:
	
\begin{corollary}\label{Cor: local}
	 For $n\in \N$, $n\geq 3$, let $\Omega \subset \R^n$ be a bounded open set with Lipschitz boundary $\p \Omega$. Let $\sigma_j\in C^2(\overline{\Omega})$ satisfy \eqref{ellipticity} with $\sigma_1=\sigma_2$ on $\p \Omega$, and $F_j \in L^2(\Omega)$, for $j=1,2$. Consider  $\Lambda_{\sigma_j,F_j}$ to be the local (full) DN map of 
	  \begin{equation}\label{equ: local j=1,2}
	 	\begin{cases}
	 		-\nabla \cdot \LC \sigma_j \nabla v_j \RC  = F_j & \text{ in }\Omega, \\
	 		v_j=g &\text{ on }\p\Omega,
	 	\end{cases}
	 \end{equation}
	 for $j=1,2$. Suppose 
	 \begin{equation}\label{DN map same local}
	 	\Lambda_{\sigma_1,F_1}(g)=\Lambda_{\sigma_2,F_2}(g), \text{ for any }g\in H^{1/2}(\p \Omega),
	 \end{equation}
	 then $\sigma_1=\sigma_2$ in $\Omega$, and there exists $\varphi \in C^2_0(\Omega)$ such that 
	 \begin{equation}\label{gauge}
	 	F_2=F_1-\nabla \cdot \LC \sigma_1 \nabla \varphi\RC \text{ in }\Omega.
	 \end{equation}
\end{corollary}

The above result states that the condition \eqref{DN map same local} can imply the uniqueness of the scalar conductivity. However the source $F$ cannot be determined due to the natural gauge invariance \eqref{gauge}.

	\section{Proof of Theorem \ref{thm: main}}\label{sec: proof}
	
Now, we can prove Theorem \ref{thm: main}.

\begin{proof}[Proof of Theorem \ref{thm: main}]
	We split the proof into two steps:
	
	\medskip
	
	{\it Step 1. Unique determination of the leading coefficient.}
	
	\medskip
	
 \noindent Let $u_j^{(f)}, u_j^{(0)}\in H^s(\R^n)$ be the solutions to 
 	\begin{equation}
 	\begin{cases}
 		(-\nabla \cdot \sigma_j \nabla )^s u_j^{(f)}=F_j &\text{ in }\Omega, \\
 		u_j^{(f)}=f &\text{ in }\Omega_e,
 	\end{cases}
 \end{equation}
 and 
 \begin{equation}\label{equ: main j=12 f=0}
 	\begin{cases}
 		(-\nabla \cdot \sigma_j \nabla )^s u_j^{(0)}=F_j &\text{ in }\Omega, \\
 		u_j^{(0)}=0 &\text{ in }\Omega_e,
 	\end{cases}
 \end{equation}
 respectively, for $j=1,2$. Let $\mathbf{u}_j^{(f)}:=u_j^{(f)}-u_j^{(0)}\in H^s(\R^n)$, then $\mathbf{u}^{(f)}_j$ be the solution to the nonlocal elliptic equation
 \begin{equation}\label{equ: main j=12 F=0}
	\begin{cases}
		(-\nabla \cdot \sigma_j \nabla )^s \mathbf{u}_j^{(f)}=0 &\text{ in }\Omega, \\
		\mathbf{u}_j^{(f)}=f &\text{ in }\Omega_e,
	\end{cases}
\end{equation}
 for $j=1,2$. Then the well-posedness of \eqref{equ: main j=12} and condition \eqref{same DN map} imply that 
 \begin{equation}
 \begin{split}
 		\mathbf{\Lambda}_{\sigma_1}^s(f)&:= \Lambda_{\sigma_1,F_1}^s(f) -  \Lambda_{\sigma_1,F_1}^s(0) \\ 
 		&\ =\Lambda_{\sigma_2,F_2}^s(f) -  \Lambda_{\sigma_2,F_2}^s(0) :=\mathbf{\Lambda}_{\sigma_2}^s(f), \text{ for any } f\in C^\infty_c(W),
 \end{split}
 \end{equation}
 where $\mathbf{\Lambda}_{\sigma_j}^s$ stands for the DN map for \eqref{equ: main j=12 F=0}, for $j=1,2$.
 Combining with the analysis in Section \ref{subsection:local-nonlocal}, the nonlocal DN map $\mathbf{\Lambda}_{\sigma_j}^s$  determines the local DN map $\mathbf{\Lambda}_{\sigma_j}$, where $\mathbf{\Lambda}_{\sigma_j}$ is the (local) DN map of  
 \begin{equation}\label{equ: local j=12 F=0}
 	\begin{cases}
 		-\nabla \cdot \LC  \sigma_j \nabla   \mathbf{v}_j^{(f)}\RC  =0 &\text{ in }\Omega, \\
 		\mathbf{v}_j^{(f)}=\left. \int_0^\infty \wt u_j^{(f)} (x,y)\, dy\right|_{\p \Omega} &\text{ on }\p \Omega,
 	\end{cases}
 \end{equation}
 for $j=1,2$.
 Here $\wt u_j^{(f)}=\wt u_j^{(f)}(x,y)$ is a solution of the extension problem \eqref{extension problem} as $\wt \sigma=\wt \sigma_j$ and $u=u_j^{(f)}$, for $j=1,2$.

 On the other hand, by the condition \eqref{same DN map}, combined with $\sigma_1=\sigma_2=1$ in $\Omega_e$, similar proof to \cite[Theorem 1.2]{CGRU2023reduction} implies that $\wt u_1^{(f)}=\wt u_2^{(f)}$ in $\Omega_e \times (0,\infty)$. 
 This implies that 
 \begin{equation}
 	 \mathbf v_1^{(f)}(x)=\int_0^\infty y^{1-2s}\wt u_1^{(f)}(x,y) \, dy=\int_0^\infty y^{1-2s}\wt u_2^{(f)}(x,y)\, dy=\mathbf v_2^{(f)}(x) \text{ for } x\in \p \Omega.
 \end{equation}
 Therefore, the condition \eqref{same DN map} implies 
 \begin{equation}
 \mathbf{\Lambda}_{\sigma_1} \LC \left. \mathbf v_1^{(f)}\right|_{\p \Omega}\RC = \mathbf{\Lambda}_{\sigma_2}\LC \left. 	\mathbf v_2^{(f)}\right|_{\p \Omega}\RC , \text{ for any }f\in C^\infty_c(W).
 \end{equation}
 In addition, the density result \cite[Proposition 1.2]{CGRU2023reduction} shows that the Dirichlet data  $\left. \mathbf{v} _j^{(f)} \right|_{\p \Omega}$ is dense in $H^{1/2}(\p \Omega)$ for any $f\in C^\infty_c(W)$ and $j=1,2$. Thus, with $\sigma_1=\sigma_2=1$ on $\p \Omega$ at hand, the global uniqueness for scalar conductivity equation yields that $\sigma:=\sigma_1=\sigma_2$ in $\Omega$.

 \medskip
 
 {\it Step 2. Unique determination of the source.}
 
 \medskip
 
 \noindent Finally, with $\sigma=\sigma_1=\sigma_2$ in $\R^n$ at hand, then the equation \eqref{equ: main j=12} becomes 
 \begin{equation}\label{source equation}
 	\begin{cases}
 		\LC -\nabla \cdot \sigma \nabla \RC^s u_j =F_j  &\text{ in }\Omega, \\
 		u_j =f &\text{ in }\Omega_e,
 	\end{cases}
 \end{equation}
 for $j=1,2$. The condition \eqref{same DN map} implies that 
 \begin{equation}
 	\LC -\nabla \cdot \sigma \nabla \RC^s \LC u_1-u_2 \RC =u_1 -u_2 =0 \text{ in }W,
 \end{equation}
 then Proposition \ref{Prop: UCP} implies $u_1=u_2$ in $\R^n$, which implies $F_1=F_2$ in $\Omega$ by using \eqref{source equation}. This proves the assertion.
\end{proof}

		\begin{remark}
		Unlike the local case, if we want to determine coefficients and source simultaneously, we could consider  semilinear equations to break the gauge. On the other hand, due to the remarkable UCP for nonlocal equations, one can expect the uniqueness result holds for nonlocal models. In addition, one can also consider a similar problem in a nonlocal parabolic equation $\LC \p _t -\nabla \cdot \sigma \nabla \RC^s u =G(x,t)$ for $0<s<1$, to determine both $\sigma$ and $G$ in a space time cylindrical domain.
	\end{remark}

	\bigskip

	\noindent\textbf{Acknowledgments.} 
	Y.-H. Lin is partially supported by the National Science and Technology Council (NSTC) Taiwan, under the projects 111-2628-M-A49-002 and 112-2628-M-A49-003. Y.-H. Lin is also a Humboldt research fellowship for experienced researcher.

	\bibliography{refs} 
	
	\bibliographystyle{alpha}

\end{document}